\documentclass[12pt]{amsart}
\usepackage{amsmath,amsthm,amsfonts,amssymb}

\newtheorem{metatheorem}{metatheorem}[section]
\newtheorem{theorem}[metatheorem]{Theorem}

\newtheorem{lemma}[metatheorem]{Lemma}
\newtheorem{corollary}[metatheorem]{Corollary}

\newtheorem{itdefinition}[metatheorem]{Definition}
\newtheorem{itexample}[metatheorem]{Example}
\newtheorem{itremark}[metatheorem]{Remark}
\newtheorem{itquestion}[metatheorem]{Question}
\newtheorem{itproblem}[metatheorem]{Problem}

\newenvironment{example}%
    {\begin{itexample}\begin{rm}}{\end{rm}\end{itexample}}
\newenvironment{definition}%
    {\begin{itdefinition}\begin{rm}}{\end{rm}\end{itdefinition}}
\newenvironment{remark}%
    {\begin{itremark}\begin{rm}}{\end{rm}\end{itremark}}
    {\begin{itquestion}\begin{rm}}{\end{rm}\end{itquestion}}
    {\begin{itproblem}\begin{rm}}{\end{rm}\end{itproblem}}

\newcommand{\Xd}{(X, d)}

\newcommand{\M}{\mathcal{M}}

\newcommand{\Mzero}{\M_0}
\newcommand{\Mone}{\M_1}
\newcommand{\Moneplus}{\M_1^+}
\newcommand{\Mplus}{\M^+}

\newcommand{\MX}{\M(X)}
\newcommand{\MzeroX}{\M_0(X)}
\newcommand{\MoneX}{\M_1(X)}
\newcommand{\MoneplusX}{\M_1^+(X)}
\newcommand{\MplusX}{\M^+(X)}

\newcommand{\Ezero}{E_0}

\newcommand{\Tzero}{T_0}
\newcommand{\Ttildezero}{\widetilde{T}_0}

\newcommand{\Imu}{I(\mu)}
\newcommand{\Inu}{I(\nu)}
\newcommand{\Imunu}{I(\mu, \nu)}

\newcommand{\N}{\mathbb{N}}
\newcommand{\R}{\mathbb{R}}

\newcommand{\im}{\mathop{\rm im}\nolimits}
\newcommand{\sign}{\mathop{\rm sign}\nolimits}
\newcommand{\supp}{\mathop{\rm supp}\nolimits}

\newcommand{\mbar}{M}

\newcommand{\ip}[2]{( #1 \mid #2 )}
\newcommand{\bigip}[2]{\bigl( #1 \bigm| #2 \bigr)}

\newcommand{\RA}{$\Rightarrow$}

\newcommand{\LRA}{$\Leftrightarrow$}

\newcommand{\wstar}{\mbox{weak-$*$}}

\newcommand{\ts}{\textstyle}

\allowdisplaybreaks

\begin{document}

\title
{Distance Geometry in Quasihypermetric Spaces.~I}

\author{Peter Nickolas}
\address{School of Mathematics and Applied Statistics,
University of Wollongong, Wollongong, NSW 2522, Australia}
\email{peter\_\hspace{0.8pt}nickolas@uow.edu.au}
\author{Reinhard Wolf}
\address{Institut f\"ur Mathematik, Universit\"at Salz\-burg,
Hellbrunnerstrasse~34, A-5020 Salz\-burg, Austria}
\email{Reinhard.Wolf@sbg.ac.at}



\date{}

\thanks{%
The authors are grateful for the financial
support and hospitality of the University of Salzburg and the Centre
for Pure Mathematics in the School of Mathematics and Applied 
Statistics at the University of Wollongong.}

\begin{abstract}
Let $\Xd$ be a compact metric space and let
$\MX$ denote the space of all finite signed Borel measures on~$X$.
Define $I \colon \MX \to \R$ by
\[
\Imu = \int_X \! \int_X d(x,y) \, d\mu(x) d\mu(y),
\]
and set
$
\mbar(X) = \sup \Imu,
$
where $\mu$ ranges over the collection of signed measures in $\MX$
of total mass~$1$.

The metric space $\Xd$ is \textit{quasihypermetric} if
for all $n \in \N$, all $\alpha_1, \ldots, \alpha_n \in \R$
satisfying
$\sum_{i=1}^n \alpha_i = 0$ and all $x_1, \ldots, x_n \in X$,
one has
$\sum_{i,j=1}^n \alpha_i \alpha_j d(x_i, x_j) \leq 0$.
Without the quasihypermetric property $\mbar(X)$ is infinite,
while with the property a natural semi-inner product structure
becomes available on $\Mzero(X)$,
the subspace of $\MX$ of all measures of total mass~$0$.
This paper explores: operators and functionals
which provide natural links between the metric structure
of~$\Xd$, the semi-inner product space structure of $\Mzero(X)$
and the Banach space~$C(X)$ of continuous real-valued functions
on~$X$; conditions equivalent to the quasihypermetric property;
the topological properties of $\Mzero(X)$ with the topology
induced by the semi-inner product, and especially
the relation of this topology to the \wstar{} topology
and the measure-norm topology on $\Mzero(X)$;
and the functional-analytic properties of $\Mzero(X)$
as a semi-inner product space, including the question
of its completeness.
A~later paper [Peter Nickolas and Reinhard Wolf,
\emph{Distance Geometry in Quasihypermetric Spaces.~II}]
will apply the work of this paper to a detailed analysis
of the constant $\mbar(X)$.
\end{abstract}

\maketitle

\section{Introduction}
\label{Introduction}

Let $\Xd$ be a compact metric space and let
$\MX$ denote the space of all finite signed Borel measures on~$X$.
Define $I \colon \MX \to \R$ by
\[
\Imu = \int_X \! \int_X d(x,y) \, d\mu(x) d\mu(y),
\]
and set
\[
\mbar(X) = \sup \Imu,
\]
where $\mu$ ranges over $\Mone(X)$,
the collection of signed measures in $\MX$ of total mass~$1$.
Our main aim in this paper and its sequels \cite{NW2} and~\cite{NW3}
is to investigate the properties of the geometric constant~$\mbar(X)$.

The so-called quasihypermetric property (for the definition, see below)
turns out to play an essential role in our analysis.
Indeed, we show that if $\Xd$ does not have the quasihypermetric property,
then $\mbar(X)$ is infinite, and, with the exception
of some general results, our attention is therefore mostly
confined to quasihypermetric spaces.
When $\Xd$ is a quasihypermetric space,
we introduce a semi-inner product on $\Mzero(X)$,
the subspace of all measures in $\MX$ of total mass~$0$.
The resulting semi-inner product space has interesting properties
in its own right, and is our fundamental tool for studying
the properties of~$\mbar(X)$.

In this paper, we focus largely on the analysis of this
semi-inner product space, and then in \cite{NW2} and~\cite{NW3}
we use the framework that this provides for a comprehensive
discussion of the properties of~$\mbar$.
Specifically, we explore in this paper:
\begin{enumerate}
\item[(1)]
the properties of several operators and functionals
which provide natural links between the metric structure
of~$\Xd$, the semi-inner product space structure of $\Mzero(X)$
and the Banach space $C(X)$ of continuous real-valued functions
on~$X$,
\item[(2)]
conditions equivalent to the quasihypermetric property,
\item[(3)]
the topological properties of $\Mzero(X)$ with the topology
induced by the semi-inner product, and especially
the relation of this topology to the \wstar{} topology
and the measure-norm topology on $\Mzero(X)$,
\item[(4)]
the functional-analytic properties of $\Mzero(X)$
as a semi-inner product space,
especially under the condition that $\mbar(X)$ is finite, and
\item[(5)]
the question of the completeness of $\Mzero(X)$ as a semi-inner product space.
\end{enumerate}
These items describe respectively the contents
of the five main sections of the paper.

As remarked above, the sequels \cite{NW2} and~\cite{NW3}
to this paper pursue in detail the applications of our work
here to the study of the constant $\mbar(X)$.
Further papers are also planned, in which we will study a number of
questions related to the issues raised in the first
two papers. These include the behaviour of~$\mbar$ in several
specific classes of metric spaces and the relation of~$\mbar$
to other constants appearing in distance geometry.

\subsection{Definitions and notation}
\label{chapter1}

Let $\Xd$ (abbreviated when possible to~$X$) be a compact metric space. 
The diameter of~$X$ is denoted by $D(X)$.

We denote by $C(X)$ the Banach space of all real-valued continuous functions
on~$X$ equipped with the usual supremum norm.
Further,
\begin{itemize}
\item
$\MX$ denotes the space of all finite signed Borel measures on~$X$,
\item
$\Mzero(X)$ denotes the subspace of $\MX$
consisting of all measures of total mass~$0$,
\item
$\Mone(X)$ denotes the affine subspace of $\MX$
consisting of all measures of total mass~$1$,
\item
$\Mplus(X)$ denotes the set of all positive measures in~$\MX$, and
\item
$\Moneplus(X)$ denotes the intersection of $\Mplus(X)$ and $\Mone(X)$,
the set of all probability measures on~$X$.
\end{itemize}
The support of $\mu \in \MX$ is denoted by $\supp(\mu)$.
For $x \in X$, we denote by $\delta_x \in \Moneplus(X)$
the point measure at~$x$.

Recall that the \wstar{} topology on $\MX$
is characterized by the fact that a net $\{ \mu_{\alpha} \}$
in $\MX$ converges to $\mu \in \MX$ if and only if
$\int_X f \, d\mu_{\alpha} \to \int_X f \, d\mu$ for all
$f \in C(X)$.

Each $\mu \in \MX$ has a Hahn--Jordan decomposition,
allowing us to write either $\mu = \mu^+ - \mu^-$,
where $\mu^+, \mu^- \in \Mplus(X)$ and
$\supp(\mu^+) \cap \supp(\mu^-) = \emptyset$, or, equivalently,
$\mu = \alpha\mu_{1} - \beta \mu_2$, where $\mu_1, \mu_2 \in \Moneplus(X)$,
$\alpha, \beta \geq 0$ and $\supp(\mu_1) \cap \supp(\mu_2) = \emptyset$.
We denote by $\| \cdot \|_\M$ the measure norm on $\MX$.
Since our standing assumption will be that $X$ is compact,
we have the simple expression
$\| \mu \|_\M = \mu^+(X) + \mu^-(X) = \alpha + \beta$,
for $\mu$ as above.

The Riesz representation theorem tells us that $\MX$,
equipped with the measure norm, is a Banach space
isometrically isomorphic to the space $C(X)'$, the dual space of $C(X)$.
In the following, we will freely identify signed Borel measures
with continuous linear functionals, writing as convenient either
$\mu(f)$ or $\int_X f \, d\mu$ when $f \in C(X)$ and $\mu \in \MX$.

\bigskip

Two functionals on measures will play a central role
in this paper. If $\Xd$ is a compact metric space,
then for $\mu, \nu \in \MX$, we set
\[
\Imunu = \int_X \! \int_X d(x, y) \, d\mu(x) d\nu(y),
\]
and then we set
\[
\Imu = I(\mu, \mu).
\]
We also make use of the linear functionals $J(\mu)$ on $\MX$,
defined for each $\mu \in \MX$ by $J(\mu)(\nu) = \Imunu$
for all $\nu \in \MX$.
The functional $I(\cdot, \cdot)$ is obviously bilinear on $\MX \times \MX$,
and this immediately gives identities such as
\[
I(\mu \pm \nu) = \Imu + \Inu \pm 2 \Imunu,
\]
which we will use frequently.
It is useful to note that $I(\delta_x) = 0$.

For $\mu \in \MX$, we define the function~$d_\mu$ by
\[
d_\mu(x) = \int_X d(x, y) \, d\mu(y)
\]
for $x \in X$. Of course, $d_\mu \in C(X)$ for all $\mu$,
and we define a linear map
\[
T \colon \MX \to C(X)
\]
by setting $T(\mu) = d_\mu$ for $\mu\in \MX$.
Note that we may express the functional $I(\cdot, \cdot)$
in terms of the functions~$d_\mu$:
\[
\Imunu
 = \int_X d_\mu \, d\nu
 = \int_X d_\nu \, d\mu
 = I(\nu, \mu).
\]
We also make use of the linear map $\Tzero$, which is the restriction
of~$T$ to the subspace $\MzeroX$.

For the compact metric space $\Xd$, we define
\[
M^+(X) = \sup \bigl\{ \Imu : \mu \in \Moneplus(X) \bigr\}
\]
and
\[
\mbar(X) = \sup \bigl\{ \Imu: \mu \in \Mone(X) \bigr\}.
\]
The geometric constant $\mbar(X)$ is our main focus in this paper,
but use will be made from time to time of $M^+(X)$.

\bigskip

A metric space $\Xd$ is called \textit{quasihypermetric}
if for all $n \in \N$, all $\alpha_1, \ldots, \alpha_n \in \R$ satisfying
$\sum_{i=1}^n \alpha_i = 0$ and all $x_1, \ldots, x_n \in X$, we have
\[
\sum_{i,j=1}^n \alpha_i \alpha_j d(x_i, x_j) \leq 0.
\]

\subsection{Connections with other work}
\label{sect:pottheory}

The geometric constant~$\mbar(X)$ appeared for the first time
in the work of Alexander and Stolarsky~\cite{AandS},
who dealt with the case when $X$ is a compact subset
of euclidean space and $d$ is the usual euclidean metric.
They showed that in this case $\mbar(X)$ is always finite,
and that when the subset~$X$ itself is finite, the supremum $\mbar(X)$
is achieved for some signed measure $\mu \in \Mone(X)$,
allowing the explicit computation of $\mbar(X)$.
Further papers by Alexander, especially \cite{Alex5} and~\cite{Alex1},
carried the analysis of the euclidean case further.
Because euclidean space is quasihypermetric,
the references just cited do not give explicit emphasis to the role
of the quasihypermetric property and have little need for
the development of a general framework for the analysis.

Our interest is in the analysis of~$\mbar(X)$
in a general compact metric space~$X$,
and our primary aim in the present paper is to develop
the framework mentioned and in particular to make
explicit the role of the quasihypermetric property.
Indeed, the constant~$\mbar(X)$, which is ultimately our main interest,
is discussed in this paper only as far as is needed to do this,
and a detailed analysis of~$\mbar(X)$ itself will be taken up
in \cite{NW2}, \cite{NW3} and later papers.

Some of the ideas developed here have obvious parallels
with the ideas of potential theory. In modern accounts
of classical potential theory (see Landkof~\cite{Lan}, for example),
one deals with a space~$X$ which is a suitable region
in a euclidean space and a kernel $k(x, y)$
on $X \times X$ which is typically of the form $\|x - y\|^\alpha$
for certain values of $\alpha < 0$ which depend on the dimension
of the euclidean space (here, $\| \cdot \|$ denotes the euclidean norm).
Energy integrals $I_k(\mu) = \int \!\! \int k(x, y) \, d\mu(x) d\mu(y)$
and potentials $d_{k,\mu}(x) = \int k(x, y) \, d\mu(y)$
are then defined for signed measures~$\mu$,
paralleling our definitions above, and one seeks, for example,
to find measures~$\mu$ which minimize~$I_k(\mu)$ or which yield
a constant potential~$d_{k,\mu}$.

The classical framework may be generalized in several ways
(see Fuglede~\cite{Fug}): the space~$X$ may be replaced
by a (locally) compact Hausdorff topological space
and quite general classes of kernels~$k$ can be considered.
As discovered already by Bj{\"o}rck~\cite{Bjo}, the theory takes on
a significantly different character even in the euclidean case
if the kernel has the non-classical form $\|x - y\|^\alpha$
for $\alpha > 0$, since one then naturally seeks to maximize
rather than to minimize the corresponding generalized energy integral.
Moreover, if $X$ is not a euclidean domain,
then standard analytical techniques, especially that
of the Fourier transform, are no longer available.

For these reasons and others (relating, for example,
to the quasihypermetric constraint), one cannot expect to find
precise parallels between our results and arguments
and those of either classical or generalized potential theory,
even though the theories have global features in common
at many points.

Some of the ideas in this paper can be generalized straightforwardly
along the lines suggested by Fuglede's work. The reader can verify
easily, for example, that analogues of a number of our results
hold in the case of a continuous, symmetric kernel~$k$
on a compact Hausdorff space~$X$.
Using Fuglede's work, Farkas and Rev\'{e}sz \cite{FR1,FR2}
recently carried out a generalized potential-theoretic analysis
of the so-called rendezvous number,
another constant appearing in distance geometry
(see, for example, \cite{Stadje}, \cite{Gross}, \cite{MandN},
\cite{CMY}, \cite{Yost1}, \cite{Wol1}, \cite{Wol2}, \cite{Wol3}
and~\cite{Hin1}).

Despite the possibility of such generalization, however,
our discussion here takes place exclusively in the setting
of a compact metric space~$X$ and its metric~$d$,
because our motivation is essentially geometric:
the analysis of the geometric properties of~$X$ and related structures,
and especially the geometric constant~$\mbar(X)$.

\section{Properties of the Mappings $T$ and $I$}
\label{sect:propsofTandI}

Recall from section~\ref{chapter1} that when $\Xd$ is a compact metric space,
$T \colon \MX \to C(X)$ is the linear map defined by $T(\mu) = d_\mu$
for $\mu \in \MX$. We denote the image of~$T$ by $\im T$.

\begin{theorem}
\label{newtheorem}
Let $\Xd$ be a compact metric space. Then $\dim(\im T)$ is finite
if and only if\/ $X$ is finite.
\end{theorem}

For the proof of the theorem, we need the following lemma.
If $S$ is any subset of a linear space, we write $[S]$
(omitting set braces if appropriate) to denote the linear hull of~$S$.

\begin{lemma}
\label{newlemma}
Let $\Xd$ be a compact metric space. Then we have the following.
\begin{enumerate}
\addtolength{\itemsep}{1mm}
\item
If $i \colon X \to C(X)$ is the function defined by
$i(x) = d_{\delta_x}$ for $x \in X$,
then $\| i(x) - i(y) \|_\infty = d(x, y)$
for all $x, y \in X$.
\item
$\overline{\im T} = \overline{[i(x) : x \in X]}$.
\end{enumerate}
\end{lemma}

\begin{proof}
Since $d_{\delta_x}(y) = d(x, y)$ for all $x, y \in X$,
the first statement is an easy consequence of the triangle inequality.

\smallskip
To prove the second statement, assume that $\mu \in \M(X)$
is such that $d_\mu \notin \overline{[i(x) : x \in X]}$.
Then by the Hahn--Banach theorem, there exists $\nu \in \M(X)$
such that $\nu(d_\mu) = 1$ and $\nu(i(x)) = 0$ for all $x \in X$.
But then $d_\nu(x) = 0$ for all $x \in X$,
while $\mu(d_\nu) = \nu(d_\mu) = 1$, a contradiction.
Therefore, $\overline{\im T} \subseteq \overline{[i(x) : x \in X]}$,
and since the reverse inclusion clearly holds, the proof is complete.
\end{proof}

\begin{proof}[Proof of Theorem~\ref{newtheorem}]
If $X$ is finite, then of course $\dim(\im T)$ is finite.

Let us assume that $\dim(\im T) = n$ for some integer~$n \geq 0$.
It is easy to see that if $n = 0$, then $X$ is a one-point space,
so we can assume that $n \geq 1$.
By Lemma~\ref{newlemma}, there are $x_1, \ldots, x_n \in X$
such that $\im T = [i(x_1), \ldots, i(x_n)]$,
and so for every $x \in X$, there exists a unique
$\lambda(x) = (\lambda_1(x), \ldots, \lambda_n(x)) \in \R^n$
such that $i(x) = \lambda_1(x) i(x_1) + \cdots + \lambda_n(x) i(x_n)$.
It follows that $d(x,y) = \sum_{i=1}^n \lambda_i(x) d(x_i, y)$
for all $x, y \in X$, and so we have
\[
d(x, y)
 = d(y, x)
 = \sum_{j=1}^n \lambda_j(y) d(x, x_j)
 = \sum_{i,j=1}^n \lambda_i(x) \lambda_j(y) d(x_i, x_j).
\]

Define an $n \times n$ matrix $A = (a_{i,j})$
by setting $a_{i,j} = -(1/2) d(x_i, x_j)$ for all $i$ and~$j$,
and view~$A$ as a bounded linear operator on the euclidean
space $\R^n$. It follows that
\[
d(x, y) = \bigip{A(\lambda(x) - \lambda(y))}{\lambda(x) - \lambda(y)}.
\]
Now by the Cauchy--Schwarz inequality, we have
\[
d(x, y)
 \leq \bigl\| A(\lambda(x) - \lambda(y) )\bigr\|
      \cdot
      \bigl\| \lambda(x) - \lambda(y) \bigr\|
 \leq \| A \| \cdot \bigl\| \lambda(x) - \lambda(y) \bigr\|^2.
\]
To estimate $\| \lambda(x) - \lambda(y) \|$,
define $\phi_j \colon \im T \to \R$ by setting
\[
\phi_j \Bigl( \sum_{i=1}^n \beta_i i(x_i) \Bigr) := \beta_j
\]
for $j = 1, \ldots, n$. Since $\phi_j$ is linear
and $\dim(\im T) = n < \infty$, we know that
$\phi_j$ is bounded. Hence, for all $x, y \in X$,
we have
\begin{eqnarray*}
\bigl| \lambda_j(x) - \lambda_j(y) \bigr|
 & = &
\bigl| \phi_j(i(x)) - \phi_j(i(y)) \bigr| \\
 & = &
\bigl| \phi_j(i(x) - i(y)) \bigr| \\
 & \leq &
\| \phi_j\| \cdot \bigl\| i(x) - i(y) \bigr\|_\infty \\
 & = &
\| \phi_j\| \cdot d(x, y),
\end{eqnarray*}
by Lemma~\ref{newlemma}.
Hence for $K := \max_j \| \phi_j \|$, we have
$| \lambda_j(x) - \lambda_j(y) | \leq K \cdot d(x, y)$
for all $x, y \in X$ and for all $j = 1, \ldots, n$,
and therefore
$\| \lambda(x) - \lambda(y) \|^2 \leq n K^2 d(x, y)^2$
for all $x, y \in X$. Combining our inequalities, we obtain
$d(x, y) \leq \| A \| n K^2 d(x, y)^2$ for all $x, y \in X$,
and hence we have $d(x, y) \geq 1/(n \| A \| K^2)$
for all distinct $x, y \in X$.
Since $X$ is compact, we conclude that $X$ is finite.
\end{proof}

\begin{remark}
\label{newremark}
We note that Theorem~\ref{newtheorem} does not in general hold
if the metric property of~$d$ is weakened.
For $n \geq 2$, let $S^{n-1}$ denote the euclidean unit sphere in~$\R^n$,
let $X$ be a compact subset of~$S^{n-1}$ and let $d(x, y) = \|x - y\|^2$
for all $x, y \in X$, where $\| \cdot \|$ is the euclidean norm.
Also, for $k = 1, \ldots, n$, define $f_k \in C(X)$
by $f_k(x) := \| x - e_k \|^2$, where $e_k$ denotes
the $k$th canonical unit vector in~$\R^n$.
Then defining $T$ and~$i$ formally as earlier
(though $d$ now may not be a metric),
we see easily that for each $x = (x_1, \ldots, x_n) \in S^{n-1}$
\[
i(x)
 =
2 \Bigl(1 - \sum_{k=1}^n x_k \Bigr) \cdot \underline{1}
 + \sum_{k=1}^n x_k f_k,
\]
where $\underline{1}$ denotes the constant function
$\underline{1}(y) := 1$ for all $y \in X$.
But it is clear that Lemma~\ref{newlemma} part~(2) still holds,
and so we have
$\im T \subseteq [\underline{1}, f_1, \dots, f_n]$,
and it follows that $\dim(\im T) \leq n+1 < \infty$.
While the function~$d$ is non-negative and symmetric,
and $d(x, y) = 0$ if and only if $x = y$, however,
it follows from a theorem of Danzer and Gr{\"u}nbaum~\cite{D&G}
that $d$ cannot satisfy the triangle inequality if $X$ has
more than $2^n$ elements. Thus the forward implication of Theorem~\ref{newtheorem} fails for every infinite choice of~$X$.
\end{remark}

\begin{theorem}
\label{2.18}
Let $\Xd$ be a compact metric space.
Then $T$ is injective if and only if\/ $\im T$ is dense in $C(X)$.
\end{theorem}

\begin{proof}
Assume that $\im T$ is not dense in $C(X)$.
Then by the Hahn--Banach theorem, there exists $\mu \neq 0$
in $\MX$ such that $\mu(d_\nu) = 0$ for all $\nu \in \MX$.
Therefore $0 = \mu (d_\nu) = \nu (d_\mu)$ for all $\nu \in \MX$,
and so $d_\mu = 0$. Hence $T$ is not injective.
On the other hand, assume that $\im T$ is dense in $C(X)$
and that $d_\mu = 0$ for some $\mu \in \MX$.
Now $d_\mu = 0$ implies $\nu (d_\mu) = 0$
for all $\nu \in \MX$, and therefore $0 = \nu (d_\mu) = \mu (d_\nu)$
for all $\nu \in \MX$. Then, since $\im T$ is dense in $C(X)$,
we have $\mu = 0$, and $T$ is injective.
\end{proof}

We now discuss the continuity of the functionals $I(\cdot)$ and
$I(\cdot, \cdot)$ on $\MX$ and $\MX \times \MX$, and on various subsets.
We omit the straightforward proofs of the first two results,
the second of which generalizes parts of the statement and proof
of Lemma~1 of~\cite{Wol1}.

\begin{theorem}
\label{sepcont}
If $\Xd$ is a compact metric space and $\MX$ is given the \wstar{}
topology, then the functional $I(\cdot, \cdot)$ on $\MX \times \MX$
is separately continuous in each variable.
\end{theorem}

\begin{theorem}
\label{contoncompact}
Let $\Xd$ be a compact metric space and let $\MX$ be given the \wstar{}
topology. Then the functional $I(\cdot)$ is continuous on
any subset of $\MX$ which is $\| \cdot \|_\M$-bounded.
\end{theorem}

\begin{corollary}
\label{cont2}
The functional $I(\cdot)$ is \wstar{} sequentially continuous on $\MX$.
\end{corollary}

\begin{proof}
Suppose that $\mu_n \to \mu$ is
a \wstar{} convergent sequence in $\MX$.
Then since $\mu_n(f) \to \mu(f)$ for each $f \in C(X)$,
the set $\bigl\{ | \mu_n(f) | : n \in \N \bigr\}$
is bounded for each $f \in C(X)$,
and it follows from the Banach--Steinhaus (or uniform boundedness)
theorem that the set $\bigl\{ \| \mu_n \|_\M : n \in \N \bigr\}$
is also bounded. It now follows from Theorem~\ref{contoncompact}
that $I(\mu_n) \to I(\mu)$, and so $I(\cdot)$ is sequentially continuous.
\end{proof}

\begin{corollary}
\label{cont1&3&4}
\mbox{}
\begin{enumerate}
\addtolength{\itemsep}{1mm}
\item
The functional $I(\cdot)$ is \wstar{} continuous on~$\MplusX$
\textup{(}and hence in particular on~$\MoneplusX$\textup{)}.
\item
When $X$ is finite, the functional $I(\cdot)$ is \wstar{} continuous on $\MX$.
\end{enumerate}
\end{corollary}

\begin{proof}
Both parts follow from Corollary~\ref{cont2}, using for part~(1)
the fact that the subset~$\MplusX$ of positive measures in~$\MX$
is metrizable (see~\cite[Theorem~12.10]{Cho1})
and for part~(2) the obvious fact that when $X$ is finite
$\MX$ is metrizable (see also \cite[Theorem~16.9]{Cho1}).
\end{proof}

Part~(1) in the case of~$\MoneplusX$ was observed earlier in~\cite{Wol1}.

\begin{remark}
\label{chadremark}
It is useful to note that the identity
\[
\Imunu = \ts\frac{1}{2} \bigl( I(\mu+\nu) - \Imu - \Inu \bigr)
\]
allows information about the continuity of $I(\cdot, \cdot)$
to be deduced from information about the continuity of~$I(\cdot)$
(this was pointed out to the first author by Ben Chad).
Hence Theorem~\ref{contoncompact} and Corollaries \ref{cont2}
and~\ref{cont1&3&4} extend in an obvious way to the functional~$I(\cdot, \cdot)$.
\end{remark}

We now establish a negative result about the continuity
of the functionals~$I$, which shows in particular that
significantly stronger positive results than those above
are impossible.

\begin{theorem}
\label{discontatone}
Let $\Xd$ be an infinite compact metric space.
Then the functionals $I(\cdot)$ and $I(\cdot, \cdot)$
are \wstar{} discontinuous everywhere.
\end{theorem}

\begin{proof}
We use here some ideas from exercise~2, \S4, Chap.~3 of~\cite{Bou}.
We define a net of pairs of measures in $\MX \times \MX$.
For index set, we take the set~$A$ of all finite subsets of $C(X)$,
directed by set inclusion. Consider a fixed collection
$\{f_1, \ldots, f_n\} \in A$, where $f_1, \ldots, f_n$ are distinct.
Then by Theorem~\ref{newtheorem}, there exists $\mu \in \MX$ such that
$d_\mu$ is not in the linear span of $\{f_1, \ldots, f_n\}$.
We may clearly assume that $\| \mu \|_\M = 1$.
By the Hahn--Banach theorem, there exists $\nu \in \MX$
such that $\nu(f_i) = 0$ for $i = 1, \ldots, n$
but $\nu(d_\mu) \neq 0$; that is, in our usual notation,
$\Imunu \neq 0$. We may clearly rescale~$\nu$ so that $\Imunu$
has any desired non-zero value, and it is convenient here
to assume that $\Imunu = n$.
Writing $\alpha = \{f_1, \ldots, f_n\}$, let us
denote the measures $\mu$ and~$\nu$ just found by
$\mu_\alpha$ and~$\nu_\alpha$, respectively.

We claim that the net $\{ \nu_\alpha \}$ converges \wstar{} to~$0$ in $\MX$.
Indeed, given $f \in C(X)$, we have $\{ f \} \in A$, and our choice
of $\nu_{\{f\}}$ means that $\nu_{\{f\}}(f) = 0$. Also, if $\alpha \in A$
is such that $\{ f \} \subseteq \alpha$, then $\nu_\alpha(f) = 0$,
so $\nu_\alpha(f) \to 0$ in~$\R$, as required for \wstar{} convergence.

We chose the measures $\{ \mu_\alpha \}$ so that $\| \mu_\alpha \|_\M = 1$
for all $\alpha$, so the~$\mu_\alpha$ all lie in the unit ball
of $\MX$, which by the Banach--Alaoglu theorem
(see also Corollary~12.7 of~\cite{Cho1}) is \wstar{} compact.
Hence there exists a \wstar{} convergent subnet,
say $\mu_{\alpha(\beta)} \to \mu$, of the $\mu_\alpha$.
Thus we have $(\mu_{\alpha(\beta)}, \nu_{\alpha(\beta)}) \to (\mu, 0)$,
where the convergence is \wstar{} in each coordinate.
But $\mu_\alpha$ and~$\nu_\alpha$ were chosen in such a way
that $I(\mu_\alpha, \nu_\alpha) = |\alpha|$ (the cardinality of~$\alpha$),
so it follows that the net $I(\mu_{\alpha(\beta)}, \nu_{\alpha(\beta)})$
diverges in~$\R$. That is, the functional $I(\cdot, \cdot)$
is discontinuous at $(\mu, 0) \in \MX \times \MX$.

A straightforward argument now shows that $I(\cdot, \cdot)$ is discontinuous
at all points in $\MX \times \MX$, and the observation in Remark~\ref{chadremark} then implies that $I(\cdot)$ is
discontinuous everywhere.
\end{proof}

We note the following result for later application.

\begin{corollary}
\label{discont1}
Let $\Xd$ be an infinite compact metric space.
Then the functional $I(\cdot)$, when restricted to the domain~$\Mzero(X)$,
is \wstar{} discontinuous at all points.
\end{corollary}

\begin{proof}
It is easy to show that $I(\cdot)$ is continuous at all points
of~$\MzeroX$ if and only if it is continuous at one,
so it suffices to show that $I(\cdot)$ is discontinuous at $0 \in \MzeroX$.

Assume that $I(\cdot)$ is continuous at $0 \in \MzeroX$,
and suppose that $\mu_\alpha \to 0$ for some net $\{ \mu_\alpha \}$
in $\MX$. Let $\delta$ be any fixed atomic probability measure,
and let $m_\alpha = \mu_\alpha(X)$.
Then $\mu_\alpha(X) = \int_X d\mu_\alpha \to 0$, so $m_\alpha \to 0$.
Put $\nu_\alpha = \mu_\alpha - m_\alpha \delta$,
so that $\nu_\alpha \in \MzeroX$.
Now $\nu_\alpha \to 0$ weak-$*$, since for any $f \in C(X)$
we have
$\int_X f \, d\nu_\alpha
 = \int_X f \, d\mu_\alpha - m_\alpha \int_X f \, d\delta
 \to 0$.
Hence, by assumption, $I(\nu_\alpha) \to 0$.
But
$I(\nu_\alpha)
 = I(\mu_\alpha - m_\alpha \delta)
 = I(\mu_\alpha) - 2 m_\alpha I(\mu_\alpha, \delta)$,
and $I(\mu_\alpha, \delta) \to 0$ by separate continuity
of $I(\cdot, \cdot)$, so
$\bigl| m_\alpha I(\mu_\alpha, \delta) \bigr|
 = |m_\alpha| \cdot |I(\mu_\alpha, \delta)|
 \to 0$,
and hence $I(\mu_\alpha) \to 0$.
Therefore, $I(\cdot)$ is continuous at $0 \in \MX$,
contradicting Theorem~\ref{discontatone}, and this completes the proof.
\end{proof}

\section{The Quasihypermetric Property}
\label{sect:qhm}

The quasihypermetric property is the most important metric property
considered in this paper and in \cite{NW2} and~\cite{NW3}.
In view of the fact that
our ultimate interest is the study of the geometric constant $\mbar$,
the following simple result explains why our focus is almost
exclusively on these spaces.

\begin{theorem}
\label{2.6}
If $\Xd$ is a compact non-quasihypermetric space,
then $\mbar(X) = \infty$.
\end{theorem}

\begin{proof}
If $X$ is non-quasihypermetric, then there exist $n \in \N$,
$\alpha_1, \ldots, \alpha_n \in \R$ with
$\sum_{i=1}^n \alpha_i = 0$ and $x_1, \ldots, x_n \in X$
such that $\sum_{i,j=1}^n \alpha_i \alpha_j d(x_i, x_j) > 0$.
Writing $\mu = \sum_{i=1}^n \alpha_i \delta_{x_i}$,
we therefore have $\mu \in \Mzero(X)$ and $\Imu > 0$
(see also condition~(3) in Theorem~\ref{qhmconds} below).
Now choose any $x \in X$, and define $\mu_n \in \Mone(X)$ by setting
$\mu_n = n \mu + \delta_x$ for each $n \in \N$.
Then $I (\mu_n) = n^2 \Imu + 2n d_\mu (x) \to \infty$ as $n \to \infty$,
giving the result.
\end{proof}

We record a list of conditions which are
equivalent to the quasihypermetric condition.

\begin{theorem}
\label{qhmconds}
Let $\Xd$ be a compact metric space. Then the following conditions
are equivalent.
\begin{enumerate}
\addtolength{\itemsep}{1mm}
\item
$\Xd$ is quasihypermetric.
\item
$\sum_{i,j=1}^n d(x_i, x_j) + \sum_{i,j=1}^n d(y_i, y_j)
 \leq
2 \sum_{i,j=1}^n d(x_i, y_j)$
for all $n \in \N$ and for all $x_1, \ldots, x_n, y_1, \ldots, y_n \in X$.
\item
$\Imu \leq 0$
for all $\mu \in \MzeroX$.
\item
$\Imunu^2 \leq \Imu \Inu$
for all $\mu, \nu \in \MzeroX$.
\item
$\Imu + \Inu \leq 2 \Imunu$
for all $\mu, \nu \in \MoneX$.
\item
$\Imu + \Inu \leq 2 \Imunu$
for all $\mu, \nu \in \MoneplusX$.
\item
$\frac{1}{2} \bigl(\Imu + \Inu\bigr)
\leq
I\bigl(\frac{1}{2}(\mu + \nu)\bigr)$
for all $\mu, \nu \in \MoneX$.
\item
$\frac{1}{2} \bigl(\Imu + \Inu\bigr)
\leq
I\bigl(\frac{1}{2}(\mu + \nu)\bigr)$
for all $\mu, \nu \in \MoneplusX$.
\end{enumerate}
\end{theorem}

To these, for completeness, we add the following variants
of the last two conditions.

\begin{itemize}
\addtolength{\itemsep}{1mm}
\item[(7$'$)]
\begin{it}
$\alpha\Imu + \beta\Inu
\leq
I(\alpha\mu + \beta\nu)$
for all $\mu, \nu \in \MoneX$ and all $\alpha, \beta \in \R$
such that $\alpha, \beta \geq 0$ and $ \alpha + \beta = 1$.
\end{it}
\item[(8$'$)]
\begin{it}
$\alpha\Imu + \beta\Inu
\leq
I(\alpha\mu + \beta\nu)$
for all $\mu, \nu \in \MoneplusX$ and all $\alpha, \beta \in \R$
such that $\alpha, \beta \geq 0$ and $ \alpha + \beta = 1$.
\end{it}
\end{itemize}

\begin{proof}
The proofs are for the most part straightforward,
and we show only the equivalence of (3) and~(4)
(and note also that the equivalence of (1) and~(2) is outlined
on page~2049 of~\cite{LPWZ}).

Assuming~(3), we define a semi-inner product $(\cdot \mid \cdot)$
on the space $\MzeroX$ by the formula
$
(\mu \mid \nu) = -\Imunu
$
for $\mu, \nu \in \MzeroX$;
that the semi-inner product axioms are satisfied is clear
(we will study and use this semi-inner product extensively below).
It is clear that the Cauchy--Schwarz inequality
for the semi-inner product gives~(4).
Conversely, assume~(4).
If $X$ is singleton, then (3) is immediate.
Otherwise, let $\nu$ be any element of $\MzeroX$
such that $\Inu < 0$; we may take
$\nu = \delta_x - \delta_y$
for any pair of distinct elements $x, y \in X$, for example.
Then (4) implies that $\Imu \leq 0$
for all $\mu \in \MzeroX$, giving~(3).
\end{proof}

An important and much less elementary equivalence is given
by Schoenberg~\cite{Sch3}: a separable metric space $\Xd$
is quasihypermetric if and only if the metric space $(X, d^\frac{1}{2})$
is isometrically embeddable in the Hilbert space~$\ell^2$.

The quasihypermetric property has been discovered several times;
it appears independently, for example,
in L\'evy~\cite{Lev}, Schoenberg~\cite{Sch3},
Bj\"orck~\cite{Bjo} and Kelly~\cite{Kel3},
in each case as part of a study
involving more general geometric inequalities.
The term `quasihypermetric' was introduced by Kelly~\cite{Kel3};
elsewhere, quasihypermetric spaces, or their metrics,
have been referred to as \textit{of negative type}
(see~\cite{Bret}, for example).

There are several important classes of quasihypermetric spaces.
\begin{enumerate}
\item
The euclidean spaces $\R^n$ for all $n \geq 1$.
\item
More generally, the space $\R^n$ for all~$n$, equipped with
the usual $p$-norm for $1 \leq p \leq 2$.
\item
All two-dimensional real normed spaces.
\item
All metric spaces with four or fewer points.
\item
The $n$-dimensional sphere $S^n$ in $\R^{n+1}$ for $n \geq 1$,
equipped with the great-circle metric.
\end{enumerate}

When $n \geq 3$, the space $\R^n$ equipped with
the $p$-norm for $2 < p \leq \infty$ is not quasihypermetric.

The first three classes of examples above are essentially given by
classical results from the theory of $L^1$-em\-bed\-dabil\-ity
(see~\cite{Lev, Dor, Herz}), as is the negative statement;
the case of a metric space with four points is part of
Blumenthal's `four-point theorem' (see \cite[Theorem~52.1]{Blu1});
and the case of the sphere with the great-circle metric
is given in~\cite{Kel2}. The cases of $\R^n$ with the
$p$-norm for $1 \leq p \leq 2$ and of $S^n$ with the great-circle
metric are also given by a general construction of Alexander~\cite{Alex5}
using the methods of integral geometry.

\begin{definition}
\label{sqhmdefn}
A compact quasihypermetric space $\Xd$ is said to be \textit{strictly quasihypermetric}
if $\Imu = 0$ only when $\mu = 0$, for $\mu \in \MzeroX$.
\end{definition}

Lemma~1 of~\cite{Bjo}, in this terminology, yields the following statement.

\begin{theorem}
\label{Rnstrict}
Every compact subset of\/~$\R^n$ is strictly quasihypermetric.
\end{theorem}

This fact has also been discovered independently more than once;
it is noted in~\cite{Alex3} that it is equivalent to the uniqueness
theorem for the Radon transform.
A theorem implying the weaker statement
that finite subsets of~$\R^n$ are strictly quasihypermetric was proved in~\cite{Sch2}. 

\begin{example}
\label{nonstrict}
Let $X$ be the circle $S^1$ of radius~$1$, given the arc-length
metric~$d$. Since $d$ is the one-dimensional form of the
great-circle metric, $X$ is quasihypermetric, as noted above.
We claim that $X$ is not strictly quasihypermetric. Indeed,
let $x_1$ and $y_1$ be diametrically opposite points
in~$X$. Then, if we set $\mu_1 = \delta_{x_1} + \delta_{y_1}$,
it is easy to see that the integral
$\int_X d(x, y) \, d\mu_1(x)$ has the constant value $\pi = D(X)$
for all~$y$. Hence, if a second measure $\mu_2$ is similarly
defined for a different pair of points $x_2, y_2$, and if we write
$\mu = \mu_1 - \mu_2$, then we have $0 \neq \mu \in \MzeroX$, while
$\Imu = I(\mu_1) + I(\mu_2) - 2I(\mu_1, \mu_2) = \pi +\pi - 2\pi = 0$.

The same argument shows that the sphere $S^{n-1}$
with the great-circle metric fails to be strictly quasihypermetric
for all $n > 1$.
The argument shows moreover that the subspace
$\{ x_1, y_1, x_2, y_2 \}$ of $S^1$ is a $4$-element metric space
which is quasihypermetric but not strictly quasihypermetric.
\end{example}

\begin{theorem}
\label{2.19}
Let $\Xd$ be a non-trivial compact strictly quasihypermetric space. Then
\begin{enumerate}
\item[(1)]
$T$ is injective, and
\item[(2)]
$\im T$ is dense in $C(X)$.
\end{enumerate}
\end{theorem}

\begin{proof}
By Theorem~\ref{2.18}, it suffices to show that $T$ is injective.
Suppose that $d_\mu = 0$ for some $\mu \in \MX$.
If $\mu (X) \neq 0$, define $\overline{\mu} \in \Mone(X)$ by
setting $\overline{\mu} = \mu/\mu (X)$,
and choose $\nu \in \Moneplus(X)$ with $\Inu > 0$
(we may take $\nu = \frac{1}{2} (\delta_x + \delta_y)$
for any pair of distinct elements $x, y \in X$, for example).
Then since $d_{\overline{\mu}} = 0$ and
$2I(\overline{\mu}, \nu) \geq I(\overline{\mu}) + \Inu$
(Theorem~\ref{qhmconds}), we have $\Inu \leq 0$, a contradiction.
Therefore, $\mu(X) = 0$. But since $d_\mu = 0$ implies $\Imu = 0$,
the strictly quasihypermetric assumption now implies that $\mu = 0$, as required.
\end{proof}

We note that the assumption that $X$ is non-trivial is necessary:
if $X$ is singleton, it is easy to see that $\im T$ is not dense in $C(X)$.

\begin{example}
\label{nondense}
Consider again the quasihypermetric, non-strictly quasihypermetric
space $\Xd$, where $X$ is the circle of radius~$1$
and $d$ is the arc-length metric.
It is obvious that whenever $x$ and~$x'$ are diametrically
opposite points in~$X$, we have $d(x, y) + d(x', y) = \pi$
for all $y \in X$, and integration with respect to an arbitrary
measure $\mu \in \MX$ then yields $d_\mu(x) + d_\mu(x') = \pi \mu(X)$.
But the collection of functions $f \in C(X)$ such that
$f(x) + f(x')$ is constant for all diametrically opposite pairs of
points $x$ and~$x'$ is clearly a proper closed subspace of $C(X)$,
and since it contains $\im T$, the latter is not dense in $C(X)$.
\end{example}

\begin{remark}
\label{sqhmremark}
We note that Theorem~\ref{2.19} gives a very simple proof
of Theorem~\ref{newtheorem} in the case
of a strictly quasihypermetric space~$X$.
Indeed, if $\im T$ were finite-dimensional for such a space,
then $\im T$ would be both closed and dense in $C(X)$,
and hence equal to $C(X)$, and the finite-dimensionality of $C(X)$
would then imply that $X$ was finite.
\end{remark}

\section{Topologies on $\MX$ and its Subspaces}
\label{topologies}

Let $\Xd$ be a compact quasihypermetric space. In the proof of Theorem~\ref{qhmconds},
we noted in passing that a semi-inner product $(\cdot \mid \cdot)$
can be defined on the subspace $\MzeroX$ of $\MX$ of measures of
total mass~$0$ by the formula
\[
(\mu \mid \nu) = -\Imunu
\]
for $\mu, \nu \in \MzeroX$.
When $\MzeroX$ is equipped with this semi-inner product,
we will denote the resulting semi-inner product space by $\Ezero(X)$.
We note that the associated seminorm $\|\cdot\|$ on $\Ezero(X)$
is given by
\[
\| \mu \| = \bigl[ -\Imu \bigr]^\frac{1}{2}
\]
for $\mu \in \MzeroX$.
In referring to the topology of $\Ezero(X)$,
we will from here on always mean the topology induced
by this seminorm; other topologies on~$\MzeroX$---%
specifically, the topologies induced on $\MzeroX$
by the \wstar{} topology and the measure-norm topology on $\MX$---%
will be named explicitly.

It is clear that the above semi-inner product becomes
an inner product, and $\Ezero(X)$ an inner product space,
precisely when $\Xd$ is strictly quasihypermetric.

The use of functionals such as $I(\cdot, \cdot)$
to define a (semi-)inner product structure is a standard
procedure in potential theory (see \cite{Lan} and~\cite{Fug},
for example) and has also been explored in somewhat different
settings such as the study of irregularity of distribution
(\cite{Alex4} and~\cite{Alex3}) and distance geometry (see~\cite{LSW}).

Recall the definition of the constant $\mbar(X)$:
\[
\mbar(X) = \sup \Imu,
\]
where $\mu$ ranges over $\Mone(X)$. In the case when $\mbar(X)$
is finite, there is a natural extension of the semi-inner
product on $\Ezero(X) = \Mzero(X)$ to a semi-inner product
on the collection $\M(X)$ of all signed Borel measures on~$X$.
Specifically, we define
\[
(\mu \mid \nu) = (\mbar(X) + 1) \mu(X) \nu(X) -\Imunu
\]
for $\mu, \nu \in \MX$, and note that the
semi-inner product space axioms are straightforward to check.
Further, the new semi-inner product is once again an inner product
precisely when $X$ is a strictly quasihypermetric space.
It is easy to see that the new semi-inner product is indeed
an extension of the earlier one.
When $\MX$ is equipped with the extended semi-inner product,
we will denote the resulting semi-inner product space by $E(X)$.

\begin{remark}
\label{EXremark}
We note that if the term $\mbar(X) + 1$ in the definition
is replaced by $\mbar(X) + \epsilon$, for any $\epsilon > 0$,
then the expression still defines an extension
of the earlier semi-inner product, though working
with the initially given form will suffice for our purposes here.

It is straightforward to show that the induced norms
are equivalent for all~$\epsilon$, so that in particular the
metric and topological properties of $E(X)$ are independent of~$\epsilon$.
Further, the identity mapping on $\Mzero(X)$ can be extended to
an isomorphism between the corresponding semi-inner product spaces
if and only if there exists a measure
$\mu_0 \in \Mone(X)$ such that $d_{\mu_0}$ is a constant function.
(The existence of measures of this type will play an important role
in \cite{NW2} and~\cite{NW3} in our analysis of $\mbar(X)$.)
\end{remark}

We will later make extensive use of the semi-inner product 
space $\Ezero(X)$. We begin this in the next section
of this paper, and continue it in \cite{NW2} and~\cite{NW3},
where we will relate
the structure of $\Ezero(X)$ in a detailed way to the properties
of the constant $\mbar(X)$.
In this section, however, we wish to study
some of the properties of $\Ezero(X)$ as a topological space,
especially the question of the relation between
the topology of $\Ezero(X)$ and other topologies
induced on $\MzeroX$ as a subspace of $\MX$.
The other topologies that we discuss are the topology
induced on $\MX$ and its subspaces by the measure
norm $\| \cdot \|_\M$, and the \wstar{} topology.
The question of the completeness of $\Ezero(X)$
will be discussed later, in section~\ref{completeness}.

\begin{theorem}
\label{normineq}
Let $\Xd$ be a compact quasihypermetric space.
Then for $\mu \in \MzeroX$, we have
$\| \mu \| \leq (D(X)/2)^\frac{1}{2} \| \mu \|_\M$.
\end{theorem}

\begin{proof}
Suppose that $\mu \in \MzeroX$ has Hahn--Jordan
decomposition $\mu = \mu^+ - \mu^-$. Then, since $\mu^+$
and~$\mu^-$ are positive measures, we have
\begin{eqnarray*}
\| \mu \|^2 
 & = & -\Imu \\*
 & = & -I(\mu^+ - \mu^-) \\
 & = & -I(\mu^+) - I(\mu^-) + 2 I(\mu^+, \mu^-) \\
 & \leq & 2 I(\mu^+, \mu^-) \\
 & = & 2 \int_X \! \int_X d(x, y) \, d\mu^+(x) d\mu^-(x) \\
 & \leq & 2 D(X) \mu^+(X) \mu^-(X) \\
 & \leq & (D(X)/2) \, \bigl(\mu^+(X) + \mu^-(X) \bigr)^2 \\*
 & = & (D(X)/2) \, \| \mu \|_\M^2,
\end{eqnarray*}
giving the result.
\end{proof}

\begin{corollary}
The topology of $\Ezero(X)$ is contained in the topology
induced on $\MzeroX$ by the measure norm on $\MX$.
\end{corollary}

\begin{remark}
We note that no better constant than $(D(X)/2)^\frac{1}{2}$
is in general possible in the inequality above. In any space
$\Xd$, let $x$ and~$y$ be two points in~$X$ such that
$d(x, y) = D(X)$, and set $\mu = \delta_x - \delta_y \in \MzeroX$.
Then it is easy to see that $\| \mu \|^2 = 2 d(x, y) = 2 D(X)$
and $\| \mu \|_\M = 2$, so that equality holds.

The argument above also shows, with minimal changes,
that if the support of~$\mu$ lies in a closed sphere of radius~$r$,
then $\| \mu \| \leq r^{1/2} \| \mu \|_\M$.
\end{remark}

We will prove that if $\Xd$ is a compact quasihypermetric space,
then the norm topology on $\Ezero(X)$ is incomparable
with the topology induced on $\MzeroX$ by the \wstar{}
topology on $\MX$ unless $X$ is finite.
One half of what we require is given by the following result.

\begin{theorem}
\label{seqexample}
Let $\Xd$ be an infinite compact quasihypermetric space.
Then there exists a sequence in $\MzeroX$ which
converges to\/~$0$ in $\Ezero(X)$ but does not converge in
the \wstar{} topology or in the measure-norm topology.
\end{theorem}

\begin{proof}
Since $X$ is infinite and compact, it contains
a non-trivial convergent sequence. Fix such
a sequence, say $x_n \to x$, 
in which the points~$x_n$ are all distinct from~$x$.
Write $c_n = d(x, x_n)$ for all $n \in \N$.

Define $\mu_n \in \MzeroX$ by setting
$\mu_n = c_n^{-1/3} (\delta_x - \delta_{x_n})$.
Then $\mu_n \to 0$ in $\Ezero(X)$, since
\[
\| \mu_n \|^2
 = -I(\mu_n)
 = 2 c_n^{-2/3} I(\delta_x, \delta_{x_n})
 = 2 c_n^{-2/3} d(x, x_n)
 = 2 c_n^{1/3}
 \to 0.
\]
But the argument in the proof of Corollary~\ref{cont2}
shows that if $\{ \mu_n \}$ converges \wstar{}
then the sequence of measure norms $\| \mu_n \|_\M$ must be bounded,
and since clearly $\| \mu_n \|_\M = 2 c_n^{-1/3} \to \infty$,
we conclude as required that $\{ \mu_n \}$ converges
neither \wstar{} nor in norm.
\end{proof}

\begin{corollary}
The topology of\/ $\Ezero(X)$ does not contain the \wstar{} topology
on $\MzeroX$, and is strictly weaker than the measure norm
topology on $\MzeroX$.
\end{corollary}

By Corollary~\ref{discont1}, there exists a \wstar{}
convergent net $\mu_\alpha \to 0$ in $\MzeroX$
such that $I(\mu_\alpha) \not\to 0$ in~$\R$; but
$\| \mu_\alpha \| = \bigl[ -I(\mu_\alpha) \bigr]^\frac{1}{2}$
by definition, so we have $\mu_\alpha \not\to 0$ in $\Ezero(X)$.
Thus, we have:

\begin{corollary}
\label{incomp2}
The topology of $\Ezero(X)$ is not contained in the \wstar{}
topology on $\MzeroX$.
\end{corollary}

We now have the result claimed earlier.

\begin{theorem}
\label{incomparable}
If $\Xd$ is an infinite compact quasihypermetric space,
then the topology of $\Ezero(X)$ and the \wstar{} topology
on $\MzeroX$ are incomparable.
\end{theorem}

\begin{remark}
\label{incomparableremark}
As the discussion above shows, the convergence of a net \wstar{}
in $\MzeroX$ does not imply the convergence of the net with respect
to the semi-inner product space topology of $\Ezero(X)$.
It is therefore worth noting that if $\mu_n \to \mu$
is a \wstar{} convergent sequence in $\MzeroX$,
then we also have $\mu_n \to \mu$ in $\Ezero(X)$.
Further, if $\mbar(X) < \infty$, then \wstar{} convergence of
an arbitrary sequence in $\MX$ implies its convergence
with respect to the topology of the semi-inner product space $E(X)$.
These statements can be proved straightforwardly using Corollary~\ref{cont2}
and Theorem~\ref{sepcont}.
\end{remark}

\section{$\mbar(X)$ and the Properties of $\Ezero(X)$}
\label{chapter2}

Let $\Xd$ be a compact quasihypermetric space.
As noted in section~\ref{topologies},
we can define the following semi-inner product and seminorm on $\Mzero(X)$:
\[
(\mu \mid \nu) : = -\Imunu,
\quad
\| \mu \| : = (\mu \mid \mu)^\frac{1}{2}
\]
for $\mu, \nu \in \Mzero(X)$. Recall also from section~\ref{topologies}
that $\Ezero(X)$ denotes $\Mzero(X)$ equipped with this semi-inner product,
and that $\Ezero(X)$ is an inner product space if and only if $X$
is strictly quasihypermetric. We begin by collecting
some elementary properties of $\Ezero(X)$.

\begin{lemma}
\label{2.7}
Let $\Xd$ be a compact quasihypermetric space. Then we have the following.
\begin{enumerate}
\item[(1)]
$|\Imunu|
 =
|( \mu \mid \nu )| \leq \| \mu \| \cdot \| \nu \|$ for all $\mu, \nu \in \Ezero(X)$.
\item[(2)]
$F = \{ \mu \in \Ezero(X) : \|\mu\| = 0 \}$ is a linear subspace of $\Ezero(X)$.
\item[(3)]
$F = \{\mu \in \Ezero(X) :
    (\mu \mid \nu) = 0 \mbox{ for all } \nu \in \Ezero(X)\}$.
\item[(4)]
With $(\nu_1 + F \mid \nu_2 + F): = (\nu_1 \mid \nu_2)$,
the quotient space $\Ezero(X)/F$ becomes an inner product space.
\item[(5)]
$F = \{\mu \in \Ezero(X) : d_\mu \mbox{ is a constant function}\}$.
\item[(6)]
If there exist $\varphi \in \Mone(X)$ and $c \geq 0$ such that
$|I (\varphi, \nu)| \leq c \| \nu \|$ for all $\nu \in \Ezero(X)$,
then for each $\mu \in \MX$ there exists
$c_\mu \geq 0$ such that $|\Imunu| \leq c_\mu \| \nu \|$
for all $\nu \in \Ezero(X)$.
\item[(7)]
If there exist $\mu_0 \in \Moneplus(X)$ and $c \geq 0$ such that
$|I(\mu_0, \nu)| \leq c \| \nu \|$ for all $\nu \in \Ezero(X)$,
then there exists $K \geq 0$ such that $|\Imunu| \leq K \| \nu \|$
for all $\nu \in \Ezero(X)$ and for all $\mu \in \Moneplus(X)$.
\end{enumerate}
\end{lemma}

\begin{proof}
It is well known that (1), (2), (3) and~(4) hold
in all semi-inner product spaces.

\smallskip
(5)
Let $\mu$ be in~$F$. Part (3) implies that
$I(\mu, \delta_x - \delta_y) = 0$ for all $x, y \in X$, and hence that
$d_\mu(x) = d_\mu(y)$ for all $x, y \in X$.
Conversely, if $d_\mu$ is a constant function,
then it is easy to check that $\Imu = \mu(X) = 0$, giving $\| \mu \| = 0$.

\smallskip
(6)
Consider $\varphi \in \Mone(X)$ and $c \geq 0$ such that
$|I(\varphi, \nu)| \leq c \| \nu \|$ for all $\nu \in \Ezero(X)$,
and let $\mu \in \MX$. If $\mu(X) = 0$, then the
assertion follows by~(1). If $\mu (X) \neq 0$, then we have
\begin{eqnarray*}
\left|I \left( \frac{\mu}{\mu (X)}, \nu \right) \right|
 & \leq &
\left|I \left( \frac{\mu}{\mu (X)} - \varphi, \nu \right) \right|
  + \bigl| I(\varphi, \nu) \bigr| \\
 & \leq &
\left( \left\| \frac{\mu}{\mu (X)}
  - \varphi \right\| + c \right) \cdot \| \nu \|,
\end{eqnarray*}
and hence
$\bigl| I (\mu, \nu) \bigr|
 \leq
\bigl( \| \mu - \mu(X) \varphi \| + c |\mu(X)| \bigr) \| \nu \|$
for all $\nu \in \Ezero(X)$.

\smallskip
(7)
Consider $\mu_0 \in \Moneplus(X)$ and $c \geq 0$ such
that $\bigl| I(\mu_0, \nu) \bigr| \leq c \| \nu \|$
for all $\nu \in \Ezero(X)$. Then for $\mu \in \Moneplus(X)$, we have
\begin{eqnarray*}
\bigl| \Imunu \bigr|
 & \leq &
\bigl| I(\mu - \mu_0, \nu) \bigr| + \bigl| I(\mu_0, \nu) \bigr| \\
 & \leq &
\| \mu-\mu_0 \| \cdot \| v \|  + c \| v \| \\
 &   =  &
\Bigl[ \bigl( 2I (\mu, \mu_0) - \Imu - I(\mu_0) \bigr)^\frac{1}{2} + c \Bigr]
\cdot \| \nu \| \\
 & \leq &
\Bigl[ \bigl( 2 I (\mu, \mu_0) - I(\mu_0) \bigr)^{\frac{1}{2}} + c \Bigr]
\cdot \| \nu \| \\
 & \leq &
\| \nu\| \cdot \Bigl[ \bigl( 2D(X) - I(\mu_0) \bigr)^\frac{1}{2} + c \Bigr],
\end{eqnarray*}
for all $\nu \in \Ezero(X)$.
\end{proof}

\begin{theorem}
\label{2.9.5}
Let $X$ be a compact quasihypermetric space.
If there exist $\mu \in \Mzero(X)$ and $c \neq 0$ such that
$d_\mu(x) = c$ for all $x \in X$, then
\begin{enumerate}
\item[(1)]
$X$ is not strictly quasihypermetric, and
\item[(2)]
$\mbar(X) = \infty$.
\end{enumerate}
\end{theorem}

\begin{proof}
Let $\mu \in \Mzero(X)$ and $c \neq 0$ be such that
$d_\mu(x) = c$ for all $x \in X$.

\smallskip
(1)
Clearly, $\mu \neq 0$ and $\Imu = 0$, and so $X$ is not strictly quasihypermetric.

\smallskip
(2) 
Write $\mu = \alpha\mu_{1} - \beta\mu_{2}$, where $\alpha, \beta \geq 0$
and $\mu_1, \mu_2 \in \Moneplus(X)$.
Since $\mu \neq 0$ and $\mu(X) = 0$, we have $\alpha = \beta \neq 0$.
Now $c = d_\mu(x) = \alpha d_{\mu_{1}-\mu_{2}}(x)$ for all $x \in X$.
Hence $d_{\mu_1-\mu_2}(x) = K$ for all $x \in X$,
where $K := c/\alpha \neq 0$.
For each $n \geq 1$, define $\nu_n \in \Mone(X)$ by setting
$\nu_n = n \sign K(\mu_1 - \mu_2) + \mu_2$.
Then, since $I(\mu_1 - \mu_2) = (1/\alpha^2) \Imu = 0$, we have
\[
\begin{array}{rcl}
I(\nu_n)
 & = & n^2 I(\mu_1 - \mu_2) + 2 n \sign K I(\mu_1 - \mu_2, \mu_2) + I(\mu_2) \\
 & = & 2n |K| + I(\mu_2) \\
 & \to & \infty
\end{array}
\]
as $n \to \infty$, and so $\mbar(X) = \infty$.
\end{proof}

Recall from section~\ref{topologies} that in the case when $\mbar(X)$\
is finite, there is a natural extension of the semi-inner product
on $\Ezero(X)$ to a semi-inner product on the whole of $\MX$,
which we then denote by $E(X)$, given by
\[
(\mu \mid \nu) = (\mbar(X) + 1) \mu(X) \nu(X) -\Imunu
\]
for $\mu, \nu \in \MX$.
In the following results, we find that a great deal
of extra information about the spaces and operators under
consideration becomes available under the assumption that
$\mbar(X)$ is finite.

\begin{theorem}
\label{2.10}
Let $\Xd$ be a compact quasihypermetric space.
Then the following conditions are equivalent.
\begin{enumerate}
\item[(1)]
$\mbar(X) < \infty$.
\item[(2)]
There exist $\mu \in \Mone(X)$ and $c \geq 0$
such that $|\Imunu| \leq c \, \| \nu \|$
for all $\nu \in \Ezero(X)$.
\item[(3)]
For all $\mu \in \MX$, there exists $c_\mu \geq 0$
such that $|\Imunu| \leq c_\mu \, \| \nu \|$
for all $\nu \in \Ezero(X)$.
\item[(4)]
There exists $K \geq 0$ such that $|\Imunu| \leq K \, \| \nu \|$
for all $\nu \in \Ezero(X)$ and for all $\mu \in \Moneplus(X)$.
\item[(5)]
There exists $c \geq 0$ such that $\| d_\nu\|_{\infty} \leq c \, \| \nu \|$
for all $\nu \in \Ezero(X)$.
\item[(6)]
There exists $c \geq 0$ such that
$|I(\mu_1) - I(\mu_2) |\leq c \, \| \mu_1 - \mu_2 \|$
for all $\mu_1, \mu_2 \in \Moneplus(X)$.
\end{enumerate}
\end{theorem}

\begin{theorem}
\label{2.10.5}
Let $\Xd$ be a compact quasihypermetric space, and assume that
$\mbar(X) < \infty$.
Then
\begin{enumerate}
\item[(1)]
$|\mu(X)| \leq \| \mu \|$ for all $\mu \in E(X)$, and
\item[(2)]
there exists $c \geq 0$ such that
$\| d_\mu \|_\infty \leq c \,  \| \mu \|$ for all $\mu \in E(X)$.
\end{enumerate}
\end{theorem}

Before proving these two theorems, we note a useful corollary and remark.

\begin{corollary}
\label{2.10.6}
Let $\Xd$ be a compact quasihypermetric space, and assume that
$\mbar(X)< \infty$. Then $\Ezero(X)$ is closed in $E(X)$.
\end{corollary}

\begin{remark}
\label{2.11}
If additionally $X$ is strictly quasihypermetric,
then we can reformulate Theorem~\ref{2.10} in the usual language
of normed linear spaces. Recall that each $\mu \in \MX$
defines a linear functional $J(\mu)$ on $\MX$ by
$J(\mu)(\nu) = \Imunu$ for $\nu \in \MX$.
Then for $X$ strictly quasihypermetric,
Theorem~\ref{2.10} tells us that the following conditions are equivalent.
 \begin{enumerate}
\item[(1)]
$\mbar(X)< \infty$.
\item[(2)]
$J(\mu) \colon \Ezero(X) \to \R$ is bounded for some $\mu \in \Mone(X)$.
\item[(3)]
$J(\mu) \colon \Ezero(X) \to \R$ is bounded for all $\mu \in \MX$.
\item[(4)]
$\sup \| J(\mu) \| < \infty$, where $\mu$ ranges over $\Moneplus(X)$.
\item[(5)]
The mapping $\Tzero \colon \Ezero(X) \to C(X)$ defined by
$\Tzero(\mu) = d_\mu$ for $\mu \in \Ezero(X)$ is a bounded linear operator.
\item[(6)]
The concave functional~$I$ is Lipschitz-continuous on $\Moneplus(X)$
with respect to the norm-induced metric.
\end{enumerate}
\end{remark}

We now turn to proofs of the theorems.

\begin{proof}[Proof of Theorem~\ref{2.10}]
Parts (6) and~(7) of Lemma~\ref{2.7}
imply the equivalence of conditions (2), (3) and~(4).

\smallskip
(4)~\RA~(1):
Let $\mu \in \Mone(X)$. Write $\mu$ as $\mu =
\alpha\mu_1-\beta\mu_2$, with $\alpha, \beta \geq 0$,
$\alpha-\beta = 1$ and $\mu_1, \mu_2 \in \Moneplus(X)$.
Then we have
\begin{eqnarray*}
\Imu
 & = & I(\alpha (\mu_1 - \mu_2) + \mu_2) \\
 & = & \alpha^2 I(\mu_1 - \mu_2) + 2 \alpha I(\mu_2, \mu_1 - \mu_2) + I(\mu_2).
\end{eqnarray*}
If $I(\mu_1 - \mu_2) = 0$, then by assumption we have
$|I (\mu_2, \mu_1 - \mu_2)|\leq K \, \| \mu_1-\mu_2 \| = 0$,
and so $\Imu = I (\mu_2) \leq M^+(X)$.
If $I(\mu_1 - \mu_2) < 0$, then we find
\begin{eqnarray*}
\Imu
 & = & -\| \mu_1 - \mu_2 \|^2 \alpha^2 + 2 \alpha I (\mu_2, \mu_1 - \mu_2) + I(\mu_2) \\
 & = & -\| \mu_1 - \mu_2 \|^2
       \left( \alpha - \frac{I(\mu_2, \mu_1 - \mu_2)}{\|\mu_1 - \mu_2 \|^2} \right)^2 \\
 &   & \qquad\mbox{} + \frac{I (\mu_2, \mu_1 - \mu_2)^2}{\| \mu_1 - \mu_2 \|^2} + I (\mu_2).
\end{eqnarray*}
Hence, in both cases, we have
\[
\Imu
 \leq
I \left( \mu_2, \frac{\mu_1-\mu_2}{\| \mu_1-\mu_2 \|} \right)^2 + M^+(X)
 \leq K^2 + M^+(X).
\]
Therefore, we have $\mbar(X) \leq K^2 + M^+(X) < \infty$.

\smallskip
(1)~\RA~(2):
Fix any $x \in X$, and assume that for all $n \in \N$ there exists
$\nu_n \in \Ezero(X)$ with
$|I(\delta_x, \nu_n)| > n \, \| \nu_n \|$.
Suppose that $\| \nu_n \| = 0$ for some~$n$.
By Lemma~\ref{2.7} part~(5) there exists $c \in \R$ with
$d_{\nu_n}(y) = c$ for all $y \in X$.
Since $\mbar(X) < \infty$ by assumption,
Theorem~\ref{2.9.5} implies that $c = 0$.
Hence $I(\delta_x, \nu_n) = d_{\nu_n} (x) = 0$, a contradiction.
Therefore, we have $\| \nu_n \| > 0$ for all $n \in \N$.
Now, defining $\mu_n \in \Mone(X)$ by
\[
\mu_n = \delta_x + \frac{n \sign I(\delta_x, \nu_n)}{\| \nu_n \|} \, \nu_n,
\]
we have
\[
I(\mu_n)
 = \frac{2n}{\| \nu_n \|} \bigl| I(\delta_x, \nu_n) \bigr| - n^2
 > 2n^2-n^2
 = n^2
 \to \infty
\]
as $n \to \infty$, contradicting the fact that $\mbar(X) < \infty$.
Thus we have $|I(\delta_x, \nu)| \leq c \| \nu \|$,
for some $c \geq 0$ and for all $\nu \in \Ezero(X)$.

\smallskip
(4)~\RA~(5):
By assumption, we have $|I(\delta_x, \nu)| \leq K \, \| \nu \|$,
for all $\nu \in \Ezero(X)$ and for all $x \in X$. Since
$d_\nu(x) = I(\delta_x, \nu)$ for $\nu \in \Ezero(X)$ and $x \in X$,
we are done.

\smallskip
(5)~\RA~(2):
Fixing any $x \in X$, we have
$|I(\delta_x, \nu)|
 = |d_\nu(x)| \leq \|d_\nu\|_{\infty} \leq c \, \| \nu \|$,
for all $\nu \in \Ezero(X)$.

\smallskip
(4)~\RA~(6):
Let $\mu_1, \mu_2 \in \Moneplus(X)$. Then by assumption, we have
\[
\bigl| I(\mu_1) - I(\mu_1, \mu_2) \bigr|
 = \bigl| I(\mu_1, \mu_1-\mu_2) \bigr| \leq K \, \| \mu_1- \mu_2 \|
\]
and
\[
\bigl| I(\mu_1, \mu_2)- I(\mu_2) \bigr|
 = \bigl| I(\mu_2, \mu_1 - \mu_2) \bigr| \leq K \, \| \mu_1 - \mu_2\|,
\]
and hence
$|I(\mu_1) - I(\mu_2)| \leq 2 K \, \| \mu_1 - \mu_2 \|$.

\smallskip
(6)~\RA~(1):
Let $\mu \in \Mone(X)$, and write $\mu = \alpha\mu_1 - \beta\mu_2$,
where $\alpha, \beta \geq 0$, $\alpha - \beta = 1$ and $\mu_1, \mu_2 \in \Moneplus(X)$.
Now
\begin{eqnarray*}
I (\mu)
 & = & I \left(\alpha (\mu_1 - \mu_2) + \mu_2 \right) \\
 & = & I\left( \beta (\mu_1 - \mu_2) + \mu_1 \right) \\
 & = & -\| \mu_1 - \mu_2 \|^2 \alpha^2
       + 2 \alpha I (\mu_2, \mu_1 - \mu_2) + I(\mu_2) \\
 & = & -\| \mu_1 - \mu_2 \|^2 \beta^2
       + 2 \beta I (\mu_1, \mu_1 - \mu_2) + I (\mu_1).
\end{eqnarray*}
Therefore, $I(\mu_2, \mu_1 - \mu_2) \leq 0$ implies
$\Imu \leq I (\mu_2) \leq M^+(X)$
and $I(\mu_1, \mu_1 - \mu_2) \leq 0$ implies
$\Imu \leq I (\mu_1) \leq M^+(X)$.

Now suppose that $I (\mu_2, \mu_1 - \mu_2) > 0$ and
$I (\mu_1, \mu_1 - \mu_2) > 0$.
It follows that $I (\mu_1) > I (\mu_1, \mu_2) > I (\mu_2)$.
Suppose that $\| \mu_1 - \mu_2 \| = 0$.
Now Lemma~\ref{2.7} part~(5) implies the
existence of some $\gamma \in \R$ such that
$d_{\mu_1}(x) - d_{\mu_2} (x) = \gamma$ for all $x \in X$.
Therefore, integrating, we have
$I(\mu_1) - I(\mu_1, \mu_2) = \gamma = I(\mu_1, \mu_2) - I(\mu_2)$,
which gives $I(\mu_1) - I(\mu_2) = 2 \gamma$. But by assumption
we have $|I(\mu_1) - I(\mu_2)| \leq c \| \mu_1 - \mu_2 \| = 0$,
which gives $|2 \gamma| \leq 0$, and hence $\gamma = 0$,
and it follows that $I(\mu_1) = I (\mu_1, \mu_2) = I(\mu_2)$,
a contradiction.

Hence we can assume that $I(\mu_1) > I(\mu_1, \mu_2) > I(\mu_2)$
and that $\| \mu_1 - \mu_2 \|> 0$.
Now, as in the proof of the case (4)~\RA~(1), we find
\begin{eqnarray*}
\Imu
 & \leq & \frac{I(\mu_2, \mu_1 - \mu_2)^2}{\| \mu_1 - \mu_2 \|^2} + I(\mu_2) \\
 & \leq & \frac{\bigl( I(\mu_1) - I(\mu_2) \bigr)^2}{\| \mu_1 - \mu_2 \|^2} + I(\mu_2),
\end{eqnarray*}
so by assumption we have $\Imu \leq c^2 + I(\mu_2) \leq c^2 + M^+(X)$

Therefore, in either case, we have $\mbar(X) \leq c^2 + M^+(X) < \infty$.
\end{proof}

\begin{proof}[Proof of Theorem~\ref{2.10.5}]
(1)
Consider $\mu \in E(X)$ with $\mu(X) \neq 0$. Then
\begin{eqnarray*}
\| \mu \|^2
 & = & \bigl( \mbar(X) + 1 \bigr) \mu(X)^2 - \Imu \\
 & = & \mu(X)^2 \bigl( \mbar(X) + 1 - I(\mu/\mu(X)) \bigr) \\
 & \geq & \mu(X)^2.
\end{eqnarray*}
Hence $|\mu(X)| \leq \| \mu \|$ for all $\mu \in E(X)$.

\smallskip
(2)
Let $x \in X$ and $\mu \in E(X)$. Since
$\ip{\mu}{\delta_x} = (\mbar(X) + 1) \mu(X) - I(\mu,\delta_x)$,
we have
\begin{eqnarray*}
|d_\mu(x)|
 & = & \bigl| (\mbar(X) + 1) \mu(X) - \ip{\mu}{\delta_x} \bigr| \\
 & \leq & \bigl( \mbar(X) + 1 \bigr) \|\mu\| + \|\delta_x\| \cdot \|\mu\| \\
 & = & \|\mu\| \cdot \bigl( \mbar(X) + 1 + (\mbar(X)+1)^\frac{1}{2} \bigr),
\end{eqnarray*}
and so
$\|d_\mu\|_\infty
 \leq
\|\mu\| \cdot \bigl( \mbar(X) + 1 + (\mbar(X)+1)^\frac{1}{2} \bigr)$.
\end{proof}

\begin{remark}
\label{2.12}
The constant~$c$ in part~(6) of Theorem~\ref{2.10} can be taken to be non-zero.
This is clear if $X$ is singleton. For non-trivial~$X$, suppose that $c = 0$.
Then $I(\mu_1) = I(\mu_2)$ for all $\mu_1, \mu_2 \in \Moneplus(X)$,
and since $I(\delta_x) = 0$ for all $x \in X$, we have $\Imu = 0$
for all $\mu \in \Moneplus(X)$. But for any distinct $x, y \in X$,
we have $\frac{1}{2} (\delta_x + \delta_y) \in \Moneplus(X)$,
and then we have
$I \bigl( \frac{1}{2} (\delta_x + \delta_y) \bigr) = d(x, y) = 0$,
a contradiction. Thus we can assume that $c \neq 0$.

We can therefore interpret part~(6) of the theorem as saying that
$\mbar(X) < \infty$ if and only if the following
strengthened quasihypermetric property holds:
there exists $L > 0$ such that
\[
I(\mu_1 - \mu_2) + L \cdot \bigl| I(\mu_1) - I(\mu_2) \bigr|^\frac{1}{2} \leq 0
\]
for all $\mu_1, \mu_2 \in \Moneplus(X)$.
(Note that by condition~(5) of Theorem~\ref{qhmconds},
the quasihypermetric property is equivalent to the statement that
$I(\mu_1 - \mu_2) \leq 0$ for all $\mu_1, \mu_2 \in \Moneplus(X)$.)
\end{remark}

It turns out that with the imposition of the condition that
$\mbar(X) < \infty$, the assertion of Theorem~\ref{2.19} leads to
a characterization of the strictly quasihypermetric property.

\begin{theorem}
\label{2.20}
Let $X$ be a non-trivial compact quasihypermetric space with $\mbar(X) < \infty$.
Then the following conditions are equivalent.
\begin{enumerate}
\item[(1)]
$X$ is strictly quasihypermetric.
\item[(2)]
$T$ is injective.
\item[(3)]
$\im T$ is dense in $C(X)$.
\end{enumerate}
\end{theorem}

\begin{proof}
(1)~\RA~(2) follows by
Theorem~\ref{2.19} and (2)~\LRA~(3) by Theorem ~\ref{2.18},
so it remains to show that (2)~\RA~(1).
Let $\Imu = 0$ for some $\mu \in \Mzero (X)$. Then part~(5) of
Lemma~\ref{2.7} gives us $c \in \R$ such that $d_\mu(x) = c$
for all $x \in X$. But by Theorem~\ref{2.9.5} we get $c = 0$,
since $\mbar(X) < \infty$. Hence $d_\mu = 0$, and therefore,
using the injectivity of~$T$, we have $\mu = 0$.
\end{proof}

\begin{remark}
\label{2.21}
The condition $\mbar(X) < \infty$ is necessary in Theorem~\ref{2.20}.
In Theorem~5.4 of~\cite{NW2}, we shall construct a space~$X$
which is quasihypermetric but not strictly quasihypermetric
and has $\mbar(X) = \infty$, but for which
it is easy to check that $T$ is injective.
\end{remark}

Consider the interval $[a, b]$ in~$\R$, with its usual metric.
For each $c \in [a, b]$, we clearly have $d_{\delta_c}(x) = |x - c|$
for all $x \in [a, b]$. It is straightforward to confirm
that the linear span of these functions in $C([a, b])$ is exactly
the subspace of piecewise linear continuous functions,
which is dense in $C([a, b])$,
and it follows that $\im T$ is dense in $C([a, b])$.
Since $\mbar ([a, b]) = (b - a)/2 < \infty$
(see Lemma~3.5 of~\cite{AandS} or Corollary~3.2 of~\cite{NW2}),
we have the following:

\begin{corollary}
\label{2.22}
Every compact subset of\/ $\R$ with the usual metric is strictly quasihypermetric.
\end{corollary}

We noted earlier (see Theorem~\ref{Rnstrict}) the fact
that each compact subset~$X$ of $\R^n$ is strictly quasihypermetric
for all~$n$. By Theorem~\ref{2.20}, this is equivalent to the fact
that $\im T$ is dense in $C(X)$ for each such~$X$. The fact that
the latter statement holds is a fundamental result in the theory
of radial basis functions; see \cite[Theorem~B.1]{Pow1}.

\section{Completeness}
\label{completeness}

We now address the question of the completeness of the spaces
$\Ezero(X)$ and $E(X)$, under the assumption that $\mbar(X)$ is finite.
Recall that the semi-norms on $\Ezero(X)$ and $E(X)$ become norms
precisely when $X$ is strictly quasihypermetric.

Our main result is the following (cf.~\cite[Theorem~1.19]{Lan}).

\begin{theorem}
\label{2.31}
Let $\Xd$ be a compact quasihypermetric space with $\mbar(X) < \infty$.
Then the semi-inner product space $\Ezero(X)$ is complete
if and only if $X$ is finite.
\end{theorem}

For the proof, we need the following result.
(Recall that $\Tzero \colon \MzeroX \to C(X)$ is
the restriction of the linear map~$T$ to the subspace $\MzeroX$.
Also recall that for $\mu \in \MX$, the functional $J(\mu)$ is defined by
$J(\mu)(\nu) = I(\mu , \nu)$ for $\nu \in \Ezero(X)$.)

\begin{lemma}
\label{2.32}
Let $\Xd$ be a compact quasihypermetric space with $\mbar(X) < \infty$.
Then we have the following.
\begin{enumerate}
\addtolength{\itemsep}{1mm}
\item[(1)]
The operator $\Ttildezero \colon \Ezero(X)/F \to C(X)$
defined by $\Ttildezero(\mu + F) = \Tzero(\mu)$
for $\mu \in \Ezero(X)$ is well defined and compact.
\item[(2)]
The adjoint operator $\Ttildezero' \colon \MX \to (\Ezero(X)/F)'$
is given by $\Ttildezero'(\mu) (\nu + F) = -\ip{\mu}{\nu}$
for all $\mu$ in $\MX$ and $\nu \in \Ezero(X)$.
\item[(3)]
$\dim \Ezero(X)/F < \infty$ if and only if $X$ is finite.
\item[(4)]
$\Ezero(X)/F$ is complete if and only if $\dim \Ezero(X)/F < \infty$.
\end{enumerate}
\end{lemma}

\begin{proof}
(1)
Suppose that $\mu_1 + F = \mu_2 + F$
for some $\mu_1, \mu_2 \in \Ezero(X)$.
Then $\mu_1 -\mu_2 \in F$, and by Lemma~\ref{2.7} part~(5)
and Theorem~\ref{2.9.5}, we conclude that $d_{\mu_1-\mu_2} \equiv 0$,
and so $\Ttildezero(\mu_1 + F) = \Ttildezero(\mu_2 + F)$.

Let $B = \bigl\{ \nu + F \in \Ezero(X)/F : \| \nu + F \| \leq 1 \bigr\}$.
For $\nu + F \in B$, we have
\begin{eqnarray*}
\bigl| \Ttildezero(\nu + F)(x) - \Ttildezero(\nu + F)(y) \bigr|
 & = & \bigl| d_\nu(x) - d_\nu(y) \bigr| \\
 & = & \bigl| \ip{\nu}{\delta_x - \delta_y} \bigr| \\
 & \leq & \| \nu \| \cdot \| \delta_x - \delta_y \| \\
 & = & \| \nu + F \| \cdot \bigl(2d(x, y)\bigr)^{\frac{1}{2}} \\*
 & \leq & \bigl(2d(x, y)\bigr)^{\frac{1}{2}},
\end{eqnarray*}
for all $x, y \in X$.

By Theorem~\ref{2.10} part~(5), we have, for each $\nu + F \in B$,
\[
\| \Ttildezero(\nu + F) \|_\infty
 = \| \Tzero(\nu) \|_\infty
 = \| d_\nu \|_\infty
 \leq c \| \nu \|
 = c \| \nu + F \|
 \leq c,
\]
for some constant~$c$.
The Arzel{\`a}-Ascoli Theorem now implies
that $\Ttildezero(B)$ is relatively compact
in $C(X)$, and therefore $\Ttildezero$ is compact.

\smallskip
(2)
By definition,we have
$\Ttildezero'(\mu)(\nu + F)
 = \mu(\Ttildezero(\nu + F))
 = \mu(\Tzero(\nu))
 = \mu(d_\nu)
 = \Imunu
 = -\ip{\mu}{\nu}$
for all $\mu \in \MX$ and $\nu \in \Ezero(X)$.

\smallskip
(3)
Of course, if $X$ is finite, then $\dim \Ezero(X)/F < \infty$,
so let us assume that $\dim \Ezero(X)/F = n$ for some natural number~$n$
(note that \mbox{$n = 0$} obviously implies that $X$ is a one-point space).
Thus there are $\mu_1, \ldots, \mu_n \in \Ezero(X)$ such that
\[
\Ezero(X)/F = [ \mu_1 + F, \ldots, \mu_n + F ].
\]
Now consider $\mu \in \Ezero(X)$.
Then there exist $\alpha_1, \ldots, \alpha_n \in \R$ such that
\[
\mu + F = \alpha_1(\mu_1 + F) + \cdots + \alpha_n(\mu_n + F),
\]
and we have $\mu - \sum_{i=1}^n \alpha_i \mu_i \in F$.
By Lemma~\ref{2.7} part~(5) and Theorem~\ref{2.9.5},
it follows that $d_{\mu - \sum_{i=1}^n \alpha_i \mu_i} = 0$.
Therefore, $d_\mu \in [ d_{\mu_1}, \ldots, d_{\mu_n} ]$,
and we conclude that $\im \Tzero = [ d_{\mu_1}, \ldots, d_{\mu_n} ]$.
But $\im T = [ \im \Tzero, d_{\delta_x} ]$
for each fixed $x \in X$, since
$d_\nu = d_{\nu-\nu(X)\delta_x} + \nu(X) \cdot d_{\delta_x}$
for each $\nu \in \M(X)$,
and so $\dim(\im T) < \infty$. Therefore,
$X$ is finite, by Theorem~\ref{newtheorem}.

\smallskip
(4)
Clearly $\Ezero(X)/F$ is complete if $\dim \Ezero(X)/F < \infty$,
so let us assume that $\Ezero(X)/F$ is complete.
The Riesz representation theorem, with Lemma~\ref{2.32} part~(2),
implies that $(\Ezero(X)/F)' = \Ttildezero'(\Mzero(X))$,
since
$\Ttildezero'(\mu)(\nu + F)
 = -\ip{\mu}{\nu}
 = \ip{-\mu + F}{\nu + F}$
for all $\mu, \nu \in \Ezero(X)$.
Therefore, $\Ttildezero' \colon \M(X) \to (\Ezero(X)/F)'$
is compact, since $\Ttildezero$ is compact by part~(1),
and $\im \Ttildezero' = (\Ezero(X)/F)'$, which is by assumption complete.

But it is well known (see for example Theorem~7.4 in~\cite{TayLay})
that this situation implies that $\im \Ttildezero'$
is of finite dimension, and hence $(\Ezero(X)/F)'$
is of finite dimension.
Therefore, $\Ezero(X)/F$ is of finite dimension,
and so $X$ is finite, by part~(3).
\end{proof}

\begin{corollary}
\label{completecoroll}
With the hypotheses of the lemma, $\Ezero(X)/F$ is complete
if and only if $X$ is finite.
\end{corollary}

\begin{proof}[Proof of Theorem~\ref{2.31}]
If $X$ is finite, the required conclusion is trivial,
so let us assume that $\Ezero(X)$ is complete.
Let $(\mu_n + F)_{n\geq1}$ be a Cauchy sequence in $\Ezero(X)/F$,
where $\mu_n \in \Ezero(X)$ for all~$n$.
Since
$\| \mu_n - \mu_m \| = \| (\mu_n + F) - (\mu_m + F) \|$
for all $n$ and~$m$, we conclude that $(\mu_n)_{n\geq1}$
is a Cauchy sequence in $\Ezero(X)$, and hence, by assumption,
there exists $\mu \in \Ezero(X)$ (not necessarily unique)
such that $\| \mu_n - \mu \| \to 0$ as $n \to \infty$.
Hence $\| (\mu_n + F) - (\mu + F) \| \to 0$ as $n \to \infty$.
Therefore, $\Ezero(X)/F$ is complete,
and hence, by Corollary~\ref{completecoroll}, $X$ is finite.
\end{proof}

\begin{corollary}
\label{completecoroll2}
Let $\Xd$ be a compact strictly quasihypermetric space
with $\mbar(X) < \infty$. Then the inner product space $\Ezero(X)$
is a Hilbert space if and only if $X$ is finite.
\end{corollary}

Finally, we apply an earlier result to extend Theorem~\ref{2.31}
to the space $E(X)$. Indeed, by Corollary~\ref{2.10.6},
$\Ezero(X)$ is closed in $E(X)$ when $\mbar(X) < \infty$,
so the completeness of $E(X)$ would imply the completeness
of $\Ezero(X)$, and we therefore have the following.

\begin{corollary}
\label{hilbert}
Let $\Xd$ be a compact quasihypermetric space with $\mbar(X) < \infty$.
Then the semi-inner product space $E(X)$ is complete
if and only if $X$ is finite.
\end{corollary}

There is also of course a result paralleling Corollary~\ref{completecoroll2}
for $E(X)$ in the strictly quasihypermetric case.

\providecommand{\bysame}{\leavevmode\hbox to3em{\hrulefill}\thinspace}

\end{document}